\documentclass[12pt]{article}
\usepackage{hyperref}
\usepackage{a4wide}
\usepackage{tikz}
\usepackage{amsmath,amsfonts,amssymb,latexsym,graphics,epsfig,url}
\usepackage{color}
\usepackage{amsthm,enumerate}
\usepackage[english]{babel}
\usepackage{graphicx}
\usepackage{soul}

\newcommand\blfootnote[1]{%
	\begingroup
	\renewcommand\thefootnote{}\footnote{#1}%
	\addtocounter{footnote}{-1}%
	\endgroup
}

\newtheorem{theorem}{Theorem}[section]

\newtheorem{lemma}[theorem]{Lemma}

\newtheorem{definition}[theorem]{Definition}

\DeclareMathOperator{\diag}{diag}

\DeclareMathOperator{\Irep}{Irep}
\DeclareMathOperator{\rank}{rank}
\DeclareMathOperator{\Sp}{Sp}

\def\C{\mathbb C}

\def\G{\Gamma}

\def\Re{\mathbb R}
\def\Z{\mathbb Z}

\begin{document}


\title{A note on the spectra and eigenspaces of the universal adjacency matrices of arbitrary lifts of graphs
	}
\author{C. Dalf\'o\footnote{Departament  de Matem\`atica, Universitat de Lleida, Igualada (Barcelona), Catalonia, {\tt{cristina.dalfo@matematica.udl.cat}}},
	M. A. Fiol\footnote{Departament de Matem\`atiques, Universitat Polit\`ecnica de Catalunya; and Barcelona Graduate School of Mathematics, Barcelona, Catalonia, {\tt{miguel.angel.fiol@upc.edu}}},
	 S. Pavl\'ikov\'a\footnote{Department of Mathematics and Descriptive Geometry,
	 	Slovak University of Technology, Bratislava, Slovak Republic,
	 	{\tt {sona.pavlikova@stuba.sk}}}, and
	J. \v{S}ir\'a\v{n}\footnote{Department of Mathematics and Statistics,
		The Open University, Milton Keynes, UK; and
		Department of Mathematics and Descriptive Geometry,
		Slovak University of Technology, Bratislava, Slovak Republic,
		{\tt {j.siran@open.ac.uk}}}
}
\date{}

\maketitle


\begin{abstract}
The universal adjacency matrix $U$ of a graph $\G$, with adjacency matrix $A$, is a linear combination of $A$, the diagonal matrix $D$ of vertex degrees, the identity matrix $I$, and the all-1 matrix $J$ with real 
coefficients, that is, $U=c_1 A+c_2 D+c_3 I+ c_4 J$, with $c_i\in \Re$ and $c_1\neq 0$. Thus, as particular cases, $U$ may be the adjacency matrix, the Laplacian, the signless Laplacian, and the Seidel matrix.
In this note, we show that basically the same method introduced before by the authors can be applied for determining the spectra and bases of all the corresponding eigenspaces of arbitrary lifts of graphs (regular or not).
\end{abstract}

\noindent{\it Keywords:}
Graph; covering; lift; universal adjacency matrix; spectrum; eigenspace.

\noindent{\it 2010 Mathematics Subject Classification:} 05C20, 05C50, 15A18.

\blfootnote{\begin{minipage}[l]{0.3\textwidth} \includegraphics[trim=10cm 6cm 10cm 5cm,clip,scale=0.15]{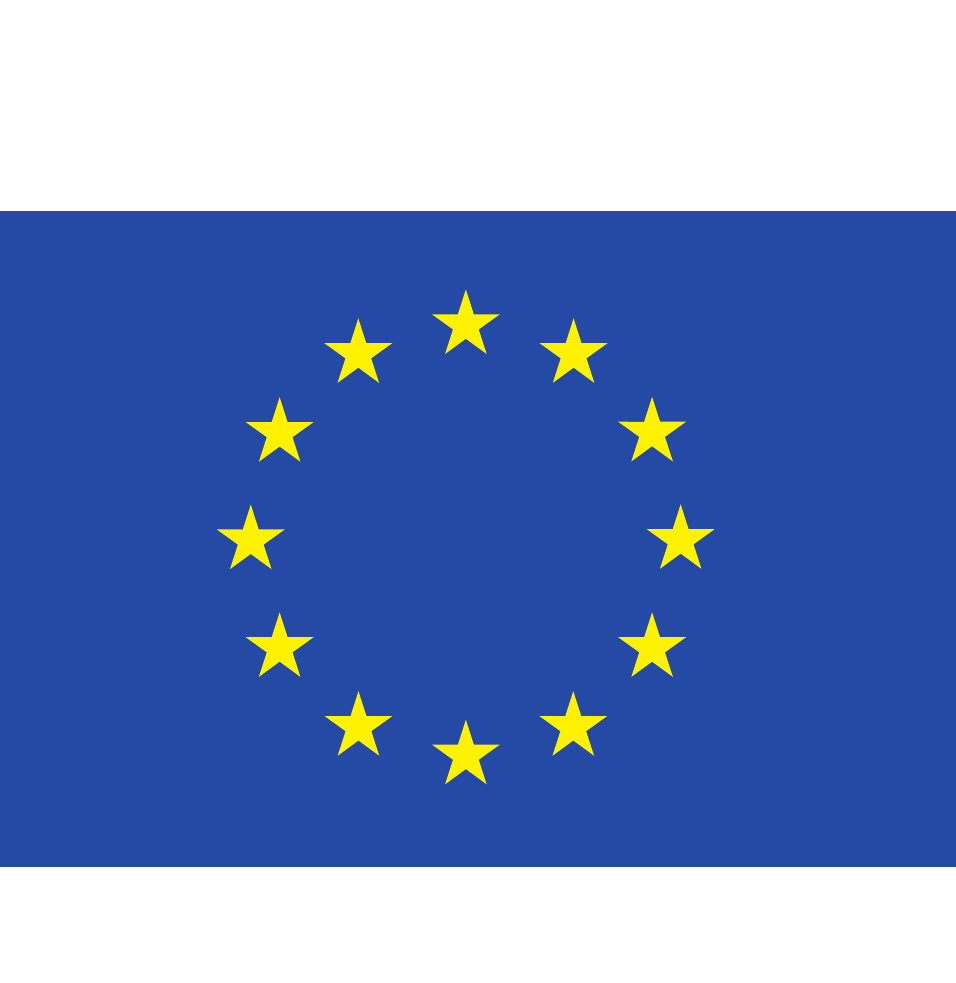} \end{minipage}  \hspace{-2cm} \begin{minipage}[l][1cm]{0.79\textwidth}
The first author has received funding from the European Union's Horizon 2020 research and innovation programme under the Marie Sk\l{}odowska-Curie grant agreement No 734922.
\end{minipage}}

\section{Introduction}\label{sec:int}

For a graph $\Gamma$ on $n$ vertices, with adjacency matrix $A$ and degree sequence $d_1,\ldots, d_n$,
the {\em  universal adjacency matrix} is defined as $U=c_1 A+c_2 D+c_3 I+ c_4 J$,  with $c_i\in \Re$ and $c_1\neq 0$, where $D=\diag(d_1,\ldots,d_n)$, $I$ is the identity matrix, and $J$ is the all-1 matrix.
See, for instance, Haemers and Omidi \cite{ho11} or
Farrugia and Sciriha  \cite{fc15}.
Thus, for given values of the coefficients $(c_1,c_2,c_3,c_4)$, the universal adjacency matrix particularizes to important matrices used in  algebraic graph theory, such as the adjacency matrix $(1,0,0,0)$, the Laplacian $(-1,1,0,0)$, the signless Laplacian $(1,1,0,0)$, and the Seidel matrix $(-2,0,-1,1)$.

Because of the interest in the study of combinatorial properties, some methods were proposed to determine the spectra of  graphs with some symmetries. This includes graphs with transitive groups (Lov\'asz \cite{l75}),
 Cayley graphs (Babay \cite{ba79}), and
 some lifts of  voltage graphs, or covers, (Godsil and  Henseland \cite{gh92}). More recently,
by using representation theory (see, for instance, Burrow~\cite{b93}, and James and Liebeck~\cite{JaLi}),    Dalf\'o, Fiol, and \v{S}ir\'a\v{n} \cite{dfs19}  considered a more general construction and derived a method for determining the spectrum of a regular lift of a `base' (di)graph equipped with an ordinary voltage assignment, or, equivalently, the spectrum of a regular cover of a (di)graph. Recall, however, that by far not all coverings are regular; a description of arbitrary graph coverings by the so-called permutation voltage assignments was given by Gross and Tucker \cite{gt77}. In  \cite{dfps19},  the authors generalized their previous results to arbitrary lifts of graphs (regular or not). They not only gave the complete spectra of lifts, but also provided bases of the corresponding eigenspaces, both in a very explicit way.
In this note, we show that (basically) the same method of \cite{dfps19} can be applied to find the spectra of lifts and their eigenspaces of a universal adjacency matrix.

This note is organized as follows. In the next section, we recall
the basic notions of permutation voltage assignments on a graph, together with their associated lifts, along with an equivalent description of these in terms of relative voltage assignments.
Section \ref{sec:main-result} deals with the main result, where the complete spectrum of the universal adjacency matrix of a relative lift graph is determined, together with bases of the associated eigenspaces. Finally, our method is illustrated by an example in Section \ref{sec:example}.

\section{General lifts of a graph}
Let $\Gamma$ be an undirected graph (possibly with loops and multiple edges), and let $n$ be a positive integer. As usual in algebraic and topological graph theory, we think of every undirected edge joining vertices $u$ and $v$ (not excluding the case when $u=v$) as consisting of a pair of oppositely directed arcs; if one of them is denoted $a$, then the opposite one is denoted $a^-$. Let $V=V(\Gamma)$ and $X=X(\Gamma)$ be the sets of vertices and arcs of $\Gamma$, respectively.

The following concepts were introduced by Gross in \cite{g74}. Let $G$ be a group. An ({\em ordinary\/}) {\em voltage assignment} on the graph $\Gamma$ is a mapping $\alpha: X\to G$ with the property that $\alpha(a^-)=(\alpha(a))^{-1}$ for every arc $a\in X$. Thus, a voltage assigns an element $g\in G$ to each arc of the graph in such a way that a pair of mutually reverse arcs $a$ and $a^{-}$, forming an undirected edge, receive mutually inverse elements $g$ and $g^{-1}$. The graph $\Gamma$ and the permutation voltage assignment $\alpha$ determine a new graph $\Gamma^{\alpha}$, called the {\em lift} of $\Gamma$, which is defined as follows. The vertex and arc sets of the lift are simply the Cartesian products $V^{\alpha}=V\times G$ and $X^{\alpha}=X\times G$, respectively. Moreover, for every arc $a\in X$ from a vertex $u$ to a vertex $v$ for $u,v\in V$ (possibly, $u=v$) in $\Gamma$, and for every $g\in G$, there is an arc $(a,g)\in X^{\alpha}$ from the vertex $(u,g)\in V^{\alpha}$ to the vertex $(v,g\alpha(a))\in V^{\alpha}$.
Notice that
if $a$ and $a^{-}$ are a pair of mutually reverse arcs forming an undirected edge of $G$, then for every $g\in G$ the pair $(a,g)$ and $(a^-,g\alpha(a))$ form an undirected edge of the lift $\Gamma^{\alpha}$, making the lift an undirected graph.
The mapping $\pi:\Gamma^{\alpha}\to \Gamma$ defined by erasing the second coordinate, that is, $\pi(u,g)=u$ and $\pi(a,g)=a$, for every $u\in V$, $a\in X$ and $g\in G$, is a (regular) {\em covering}, with its usual meaning in algebraic topology; see, for instance, Gross and Tucker \cite{gt87}.

To generate all  (not necessarily regular) graph coverings, Gross and Tucker \cite{gt77} introduced the more general concept of  permutation voltage assignment. With this aim, let
$G$ to be a subgroup of the symmetric group ${\rm Sym}(n)$, that is, a permutation group on the set $[n]=\{1,2, \ldots, n\}$. Then, a {\em permutation voltage assignment} on  $\Gamma$ is a mapping $\alpha: X\to G$ with the same symmetric property as before. So, the lift $\Gamma^{\alpha}$, with vertex set
$V^{\alpha}=V\times [n]$ and arc set $X^{\alpha}=X\times [n]$, has
an arc $(a,i)\in X^{\alpha}$ from the vertex $(u,i)$ to the vertex $(v,j)$ if and only if  $a=(u,v)\in X$ and $j=i\alpha(a)$. Notice that we write the argument of a permutation to the left of the symbol of the permutation, with the composition being read from the left to the right.

The permutation voltage assignments and the corresponding lifts can, equivalently, be described in the language of the so-called relative voltage assignments defined as follows. Let $\Gamma$ be the graph considered above, $K$ a group and $H$ a subgroup of $K$ of index $n$; and let $K/H$ denote the set of right cosets of $H$ in $K$. Furthermore, let $\beta: X\to K$ be a mapping satisfying $\beta(a^-)= (\beta(a))^{-1}$ for every arc $a\in X$; in this context, one calls $\beta$ a {\em voltage assignment in $K$ relative to $H$}, or simply a relative voltage assignment. The {\em relative lift} $\Gamma^{\beta}$ has vertex set $V^{\beta}= V\times K/H$ and arc set $X^{\beta}=X\times K/H$. The incidence in the lift is given as expected: If $a$ is an arc from a vertex $u$ to a vertex $v$ in $\Gamma$, then for every right coset $J\in K/H$ there is an arc $(a,J)$ from the vertex $(u,J)$ to the vertex $(v,J\beta(a))$ in $\Gamma^{\beta}$. We slightly abuse the notation and denote by the same symbol $\pi$ the covering $\Gamma^{\beta}\to \Gamma$ given by suppressing the second coordinate.

Under additional and very natural assumptions, there is a $1$-to-$1$ correspondence between {\em connected} lifts generated by permutation and relative voltage assignments on a connected base graph $\Gamma$. To present more details, fix a vertex $u$ in the base graph and let ${\cal W}_u$ be the collection of all closed walks in the base graph that begin and end at $u$. If $\gamma$ is either a permutation or a relative voltage assignment on $\Gamma$ in some $G\le {\rm Sym}(n)$ or in some group $K$ relative to a subgroup $H$, respectively, and if $W=a_1a_2\ldots, a_k$ is a sequence of arcs forming a walk in ${\cal W}_u$, we let $\gamma(W)=\gamma(a_1) \gamma(a_2)\ldots \gamma(a_k)$ denote the product of the voltages taken as the walk is traversed. Then, the set $\{\gamma(W);\ W\in {\cal W}_u\}$ forms a subgroup of $G$ or $K$, known as the {\em local group}, and it will be denoted by $G_u$ and $K_u$, respectively. It is well known that there is no loss of generality in the study of voltage assignments and lifts under the assumption that $G_u=G$ and $K_uH=K$, meaning that there are no `superfluous' voltages in the assignments.

We are now ready to explain the correspondence between the two kinds of lifts. First, let $\alpha$ be a permutation voltage assignment as in the previous paragraph, taking place in a permutation group $G\le {\rm Sym}(n)$. Assume that $G=G_u$ as above and, moreover, that $G$ is transitive on $[n]$; the two conditions together are equivalent to the connectivity of the lift. Then, the corresponding relative voltage group is $K=G$, with the subgroup $H$ being the stabilizer in $K$ of an arbitrarily chosen point in the set $[n]$. The corresponding relative assignment $\beta$ is simply identical to $\alpha$, but acting by right multiplication on the set $K/H$. Observe that, in this construction, the core of $H$ in $K$, that is, the intersection of all $K$-conjugates of $H$, is trivial. Conversely, let $\beta$ be a voltage assignment in a group $K$ relative to a subgroup $H$ of index $n$ and with a trivial core in $K$. Assume again that $K_uH=K$. Now, this alone guarantees the connectivity of the lift. Then, one may identify the set $K/H$ with $[n]$ in an arbitrary way, and $\alpha(a)$ for $a\in X$ is the permutation of $[n]$ induced (in this identification) by right multiplication on the set of (right) cosets $K/H$ by $\beta(a)\in K$.

We note that there are fine points in this correspondence to be considered on the `permutation' side if $G_u$ is intransitive or a proper subgroup of $G$, and on the `relative' side if $H$ is not core-free or if $K_uH$ is proper in $K$; these details are, however, irrelevant for our purposes. We also point out that a covering described in terms of a permutation voltage assignment as above is regular (that is, generated by an ordinary voltage assignment) if and only if the corresponding voltage group is a regular permutation group on the set $[n]$. Of course, this translates to the language of relative voltage assignments by the normality of $H$ in $K$. In such a case, the covering admits a description in terms of ordinary voltage assignment in the factor group $K/H$ and with voltage $H\beta(a)$ assigned to an arc $a\in X$ with original relative voltage $\beta(a)$.

\section{The spectrum of the universal adjacency matrix of a relative lift}
\label{sec:main-result}

Let $\Gamma$ be a connected graph on $k$ vertices (with loops and multiple edges allowed), and let $\beta$
be a relative voltage assignment on the arc set $X$ of $\Gamma$ in a  group
$G$ with identity element $e$, and subgroup $H$ of index $n$.
Now  we show that the  spectrum of (the universal adjacency matrix of) relative lift $\Gamma^{\beta}$  may be computed by using the same result by the authors in \cite{dfps19}. The only basic difference is the so-called base matrix, which in  our general case, must be defined as follows.

\begin{definition}
\label{B(U)}
To the pair $(\Gamma,\alpha)$ as above, we assign the $k\times k$ {\em universal base matrix} $B(U)$ defined by the sum
$$
B(U)=c_1B(A)+c_2B(D)+c_3B(I)+c_4B(J),
$$
where the matrices $B(A)+B(D)+B(I)+B(J)$ have entries as follows:
\begin{itemize}
\item
$B(A)_{uv}=\alpha(a_1)+\cdots +\alpha(a_j)$  if $a_1,\ldots,a_j$ is the set of all the arcs of $\Gamma$  from $u$ to $v$, not excluding the case $u=v$, and $B(A)_{uv}=0$ if $(u,v)\not\in X$;
\item
$B(D)_{uu}=\deg(u)e$, and $B(D)_{uv}=0$ if $u\neq v$;
\item
$B(I)_{uu}=e$, and $B(I)_{uv}=0$ if $u\neq v$;
\item
$B(J)_{uv}=\pi^0+\pi^1+\cdots+\pi^{n-1}$, where $\pi(=\pi^1)=(12\ldots n)\in Sym(n)$, for any $u,v\in V$.
\end{itemize}
Recall that $e(=\pi^0)$ stands for the identity element of $G$, and the sum must be seen as an element of the complex group algebra $\C(G)$.
\end{definition}

Let $\rho \in \Irep(G)$ be a unitary irreducible representation of $G$ of dimension $d_\rho$. For our graph $\Gamma$ on $k$ vertices, the assignment $\alpha$ in $G$ relative to $H$, and the universal base matrix $U$, let $\rho(U)$ be the $d_\rho k\times d_\rho k$ matrix obtained from $U$ by replacing every non-zero entry $(U)_{u,v} \in \C(G)$ as above by the $d_\rho\times d_\rho$ matrix $\rho(U_{u,v})$, where each element $g$ of the group is replaced by $\rho(g)$, and
the zero entries of $B$ are changed to all-zero $d_\rho\times d_\rho$ matrices. We refer to $\rho(U)$ as the {\em $\rho$-image} of the universal base matrix $U$.

In the following theorem, whose particular case of the adjacency spectrum was given by the authors in \cite{dfps19}, we use the rank of the matrix $\rho(H)= \sum_{h\in H}\rho(h)$. Also, for every $\rho\in \Irep(G)$, we consider the $\rho$-image of the universal base matrix $U$, and we let $\Sp(\rho(U))$ denote the spectrum of $\rho(U)$, that is, the multiset of all the $d_\rho k$ eigenvalues of the matrix $\rho(U)$. Finally, the notation $\rank(\rho(H))\cdot \Sp(\rho(B))$ denotes the multiset of $\rank(\rho(H))\cdot d_\rho k$ values obtained by taking each of the $d_\rho k$ entries of the spectrum $\Sp(\rho(B))$ exactly $\rank(\rho(H))$ times. In particular, if $\rank(\rho(H))=0$, the set $\rank(\rho(H))\cdot \Sp(\rho(B))=\emptyset$.

\begin{theorem}
\label{t:spec-univ}
Let $\Gamma$ be a base graph of order $k$ and let $\alpha$ be a voltage assignment on $\Gamma$ in a group $G$ relative to a subgroup $H$ of index $n$ in $G$. Then, the multiset of the $kn$ eigenvalues of the universal adjacency matrix $U^{\alpha}$ for the relative lift $\Gamma^{\alpha}$ is
\begin{equation}
\label{eq:t:spec-univ}
\Sp( U^{\alpha})=\bigcup_{\rho\in \Irep(G)}\rank(\rho(H))\cdot  \Sp(\rho(B(U))).
\end{equation}
\end{theorem}

\begin{proof}
Let $U^{\alpha}$ be the universal adjacency matrix $U=c_1A+c_2D+c_3I+c_4J$, for some constants $c_i\in \Re$, of the relative lift $\Gamma^{\alpha}$ with vertex set $V^{\alpha}=V\times G/H$, where $V$ is the vertex set of the base graph $\Gamma$, and $G/H$ is the set of right cosets of $H$ in $G$. In more detail, $U^{\alpha}$ is a $kn\times kn$ matrix indexed by ordered pairs $(u,J)\in V^{\alpha}$ whose entries are given as follows.
If there is no arc from $u$ to $u'$ for $u,u'\in V$ in $\Gamma$, then there is no arc from any vertex $(u,J)$ to any vertex $(u',J')$ for $J,J'\in G/H$ in the lift $\Gamma^{\alpha}$. So, in this case, the $(u,J)(u',J')$ entry of $U^{\alpha}$ is $c_2\deg(u)+c_3+c_4$.
In the case of existing adjacencies, if for some $g\in G$, there are exactly $t$ arcs $a_1,\ldots,a_t$ from a vertex $u$ to a vertex $v$ in $\Gamma$ that carry the voltage $g=\alpha(a_1)=\cdots =\alpha(a_t)$, then for every right coset $J\in G/H$ the $(u,J)(v,Jg)$-th entry of $U^{\alpha}$ is equal to $c_1t+c_2\deg(u)+c_3+c_4$.

From this,  the proof is basically the same as in \cite[Th. 4.1]{dfps19}, so that now we only sketch it. The steps  are the following:
\begin{enumerate}
\item
Depending on the coefficients $c_i$ with $1\le i\le 4$, consider the corresponding universal base $k\times k$ matrix
of $\Gamma$ relative to $H$, according to Definition \ref{B(U)}.
\item
For each irreducible representation $\rho\in \Irep(G)$, with dimension $d_{\rho}$,  consider the $\rho$-image $\rho(B(U))$ of the universal base matrix.
\item
Prove that every eigenvector of $\rho(B(U))$ can be ``lifted'' to at most $d_{\rho}$ eigenvectors of the corresponding universal adjacency matrix $U^{\alpha}$ of the lift $\Gamma^{\alpha}$. Thus proving that the spectrum of $\rho(B(U))$ is contained in the whole spectrum of $U^{\alpha}$.
\item
Prove that, for each $\rho\in \Irep(G)$, there are exactly $\rank(\rho(H))\cdot d_{\rho}$  independent eigenvectors of  $U^{\alpha}$ obtained at step 3, and that they also are independent of the other eigenvectors obtained from any other irreducible  representation $\rho'\in \Irep(G)$.
\item
Prove that the total number of obtained independent eigenvalues of $U^{\alpha}$, that is $\sum_{\rho\in\Irep(G)} \rank(\rho(H))\cdot d_{\rho}$, equals $n$, the number of vertices of $\Gamma^{\alpha}$, see \cite[Prop. 3.3 ]{dfps19}.
\end{enumerate}
Notice that the proof is constructive, in the sense that we determine bases for all the eigenspaces of the relative lift $\Gamma^{\alpha}$.
\end{proof}

In particular, when $H$ is the trivial subgroup, we get an ordinary voltage assignment on $G$, and  \eqref{eq:t:spec-univ} becomes
\begin{equation}
\label{eq:t:spec-univ-ord}
\Sp( U^{\alpha})=\bigcup_{\rho\in \Irep(G)}\dim_{\rho}\cdot  \Sp(\rho(B(U)))
\end{equation}
since $\rank(\rho(H))=\rank(e)=\dim(\rho)$.

\section{An example}
\label{sec:example}

\begin{figure}[t]
	\begin{center}
		\scalebox{0.78}
		{\includegraphics{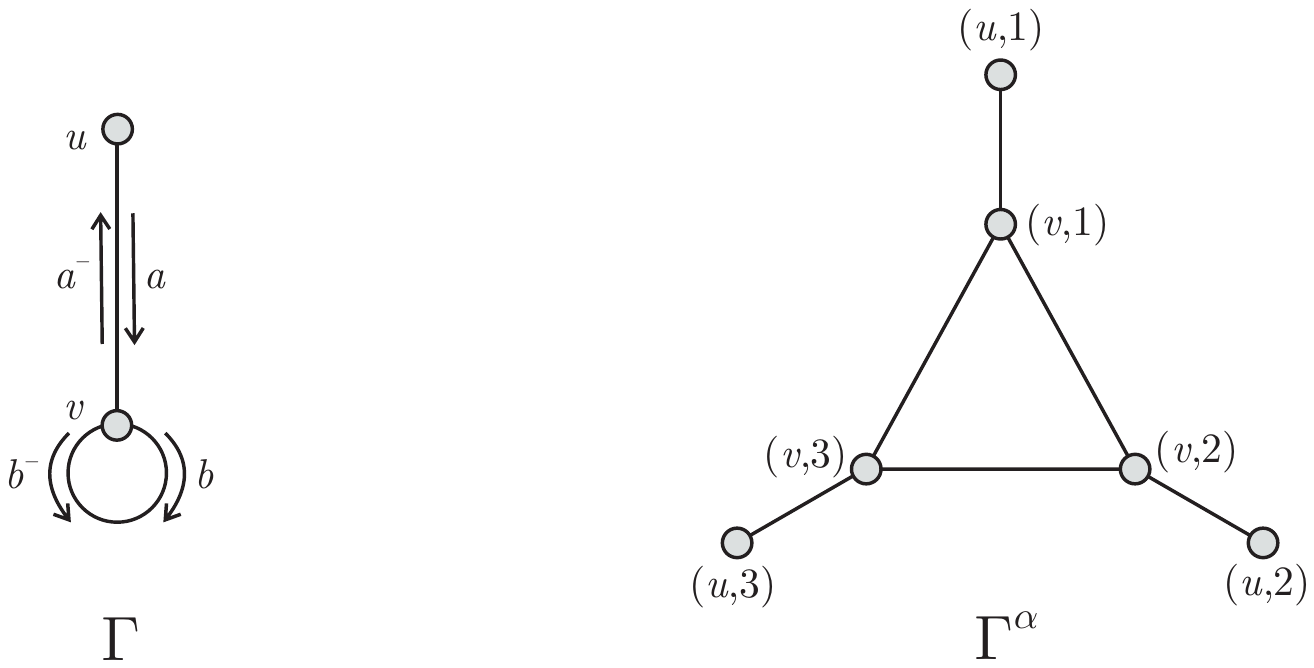}}
		\vskip-17.25cm
		\caption{A graph $\Gamma$ and its relative lift $\Gamma^{\alpha}$ in the group Sym(3). }
		\label{fig1}
	\end{center}
\end{figure}

Following the method of Theorem \ref{t:spec-univ}, we now work out some different  spectra of the relative lift shown in Figure \ref{fig1}. Referring to such a figure, we consider the base graph $\Gamma$ with vertex set $V=\{u,v\}$ and arc set $X=\{a,a^-,b,b^-\}$, where the pairs $\{a,a^-\}$, $\{b,b^-\}$ determined an edge joining $u$ to $v$, and a loop at $v$, respectively. The permutation voltage assignment $\alpha$ on $\Gamma$ in the group ${\rm Sym}(3)$ is given by $\alpha(a)=\alpha(a^-)=e$, $\alpha(b)=(123)=s$, and $\alpha(b^-)=(132)=s^2=t$. An equivalent description is to regard $\alpha$ as a relative voltage assignment, with values of $\alpha$ on arcs acting on the right cosets of $G/H$ for $H={\rm Stab}_G(1)=\{e,(23)\}$ by right multiplication. 
Then, the different base matrices with entries in the group algebra $\C(G)$ have the form

\begin{align*}
B(A) & = \left(\begin{array}{cc} 0 & \alpha(a) \\ \alpha(a^{-1}) & \alpha(b)+\alpha(b^-) \end{array}\right) = \left(\begin{array}{cc} 0 & e \\ e & s+t \end{array}\right);\\
B(D) & = \left(\begin{array}{cc} \deg(u)e & 0 \\ 0 & \deg(v)e \end{array}\right) = \left(\begin{array}{cc} e & 0 \\ 0 & 3e \end{array}\right);\\
B(I) & = \left(\begin{array}{cc} e & 0 \\ 0 & e \end{array}\right);\\
B(J) & = \left(\begin{array}{cc} e+s+t & e+s+t \\ e+s+t & e+s+t \end{array}\right).
\end{align*}

The  voltage group $G={\rm Sym}(3)=\{e,g,h,r,s,t\}$ with $g=(23)$, $h=(12)$, $r=(13)$, $s=(123)$, and $t=(132)$ has a complete set of irreducible unitary representations $\Irep(G)=\{\iota,\pi,\sigma\}$ corresponding to the symmetries of an equilateral triangle with vertices positioned at the complex points $e^{i\frac{2\pi}{3}}$, $1$, and $e^{i\frac{4\pi}{3}}$, where
\begin{align*}
\iota &: G\to \{1\},\quad \dim(\iota)=1\quad \mbox{(the trivial representation)},\\
\pi &: G\to \{\pm 1\},\quad \dim(\pi)=1\quad \mbox{(the parity permutation representation), and},\\
\sigma &:G\to GL(2,\C),\quad \dim(\sigma)=2, \mbox{ generated by the unitary matrices}
\end{align*}
\begin{equation*}
\label{eq:reprho}
\sigma(g) = \frac{1}{2}\left(\begin{array}{cc} -1 & -\sqrt{3} \\  -\sqrt{3} & 1 \end{array}\right)\qquad {\rm and} \qquad
\sigma(h) = \frac{1}{2}\left(\begin{array}{cc} -1 & \sqrt{3} \\ \sqrt{3} & 1 \end{array}\right),
\end{equation*}
whence we obtain
$$
\sigma(r)=\sigma(ghg)= \left(\begin{array}{cc} 1 & 0\\
0 & -1 \end{array}\right),\quad \sigma(s)=\sigma(gh)= \frac{1}{2}\left(\begin{array}{cc} -1 & -\sqrt{3} \\  \sqrt{3} & -1 \end{array}\right),\quad \mbox{and}
$$
$$
\sigma(t)=\sigma(hg)= \frac{1}{2}\left(\begin{array}{cc} -1 & \sqrt{3} \\  -\sqrt{3} & -1 \end{array}\right).
$$

To determine the `multiplication factors' appearing in front of the spectra in the statement of Theorem \ref{t:spec-univ}, we evaluate $\iota(H)=\iota(e)+\iota(g)=1+1=2$, of rank 1, $\pi(H)=\pi(e) + \pi(g) = 1-1=0$, of rank 0, and
\begin{equation}
\label{eq:rH}
\sigma(H)=\sigma(e) + \sigma(g) = \frac{1}{2}\left(\begin{array}{rr} 1 & -\sqrt{3} \\ -\sqrt{3} & 3\end{array}\right)
\end{equation}
of rank 1. Then, by Theorem \ref{t:spec-univ}, the spectra of the universal adjacency matrices $U^{\alpha}$ of the relative lift $\Gamma^{\alpha}$ is obtained by  the union of the sets $\Sp(\iota(U))$ and $\Sp(\sigma(U))$. Let us see some different cases.

\subsection*{Adjacency spectrum}
In this case, we only need to consider the images
of $B=B(A)$ under the representations $\iota$ and $\sigma$, which are:
\begin{equation}
\label{eq:rhoB}
\iota(B) {=}
\left(\begin{array}{cc}
 0 & 1 \\
1 & 2 \end{array}\right),\quad
\sigma(B) = \left(\begin{array}{cccc}
0 & 0 & 1 & 0 \\
0 & 0 & 0 & 1 \\
1 & 0 & -1 & 0 \\
0 & 1 & 0 & -1
\end{array}\right),
\end{equation}
with spectra $\Sp(\iota(B))=\{1\pm \sqrt{2}\}$ and $\Sp(\sigma(B))=\{[\frac{1}{2}(-1\pm\sqrt{5})]^{2}\}$. Then,
$$
\textstyle
\Sp(A^{\alpha}) = \{1\pm \sqrt{2},[\frac{1}{2}(-1\pm\sqrt{5})]^2\},
$$
where the superscripts indicate the eigenvalue  multiplicities.

\subsection*{Laplacian spectrum}
Since the Laplacian matrix can be written as $L=D-A$, we consider the base matrix
of $B(L)=B(D)-B(A)=\left(\begin{array}{cc}
 e & -e \\
 -e & 3e-s-t \end{array}\right)$
 under the representations $\iota$ and $\sigma$, which are:
\begin{equation}
\label{eq:rhoB(L)}
\iota(B(L)) {=}  \left(\begin{array}{cc}
 1 & -1 \\
-1 & 1 \end{array}\right),\quad
\sigma(B(L)) = \left(\begin{array}{cccc}
1 & 0 & -1 & 0 \\
0 & 1 & 0 & -1 \\
-1 & 0 & 4 & 0 \\
0 & -1 & 0 & 4
\end{array}\right),
\end{equation}
with spectra $\Sp(\iota(B(L)))=\{0,2\}$ and $\Sp(\sigma(B(L)))=\{[\frac{1}{2}(5\pm\sqrt{13})]^2\}$. Thus,
$$
\textstyle
\Sp(L^{\alpha}) = \{0,2,[\frac{1}{2}(5\pm\sqrt{13})]^2\}.
$$

\subsection*{Signless Laplacian spectrum}
The signless Laplacian is $Q=A+D$. So, we consider the base matrix
of $B(Q)=B(A)+B(D)=
\left(\begin{array}{cc}
e & e \\
e & 3e+s+t
\end{array}\right)$
with the representations $\iota$ and $\sigma$, as follows:
\begin{equation}
\label{eq:rho(Q)}
\iota(B(Q)) {=}
\left(\begin{array}{cc} 1 & 1 \\
 1 & 5 \end{array}\right),\quad
\sigma(B(Q)) = \left(\begin{array}{cccc}
1 & 0 & 1 & 0 \\
0 & 1 & 0 & 1 \\
1 & 0 & 2 & 0 \\
0 & 1 & 0 & 2
\end{array}\right),
\end{equation}
with spectra $\Sp(\iota(B(Q)))=\{3\pm\sqrt{5}\}$ and $\Sp(\sigma(B(Q)))=\{[\frac{1}{2}(3\pm\sqrt{5})]^2\}$. Thus,
$$
\textstyle
\Sp(Q^{\alpha}) = \{3\pm\sqrt{5},[\frac{1}{2}(3\pm\sqrt{5})]^2\}.
$$

\subsection*{Seidel spectrum}
The Seidel matrix can be defined  as $S=\overline{A}-A=J-2A-I$. Then, we consider the base matrix
of $B(S)=B(J)-2B(A)-B(I)=\left(\begin{array}{cc}
s+t & -e+s+t \\
 -e+s+t & -s-t \end{array}\right)$
under the representations $\iota$ and $\sigma$, which are:
\begin{equation}
\label{eq:rho(B(S))}
\iota(B(S)) {=}  \left(\begin{array}{cc}
2 & 1 \\
1 & -2
\end{array}\right),\quad
\sigma(B(S)) = \left(\begin{array}{cccc}
-1 & 0 & -2 & 0 \\
0 & -1 & 0 & -2 \\
-2 & 0 & 1 & 0 \\
0 & -2 & 0 & 1
\end{array}\right),
\end{equation}
with spectra $\Sp(\iota(B(S)))=\{\pm\sqrt{5}\}$ and $\Sp(\sigma(B(S)))=\{[\pm\sqrt{5}]^2\}$. Thus,
$$
\textstyle
\Sp(S^{\alpha}) = \{[\pm\sqrt{5}]^3\}.
$$

\subsection{Using ordinary voltage assignments}
Note that, in our example, the group $G$ generated by the permutation $s=(123)$ is a regular permutation on $\{1,2,3\}$, isomorphic to the cyclic group $\Z_3$. Then, the covering $\Gamma^{\beta}\to \Gamma$ is regular, and we can also construct the lift $\Gamma^{\alpha}$ from an ordinary voltage assignment on $\Z_3$.
As it is well-known, the irreducible representations $\rho_i$, for $i=0,\ldots,n-1$ of the cyclic group $\Z_n$ are
$\phi_{i}:\ 1,\omega^i,\omega^{2i},\ldots,\omega^{(n-1)i}$, for $i=0,\ldots,n-1$, where $\omega$ is a primitive $n$-th root of unity. Thus, all these representations have dimension one. In our case,
with $e=s^0$, $s$, and $s^2=s^{-1}$ being the elements of $\Z_3$, such representations are shown in Table \ref{table-Z3}.
\begin{table}[h]
\centering
\begin{tabular}{|c||c|c|c|}
\hline
$\quad g\in \Z_3$  & $e$  & $s$ & $s^2$ \\
\hline\hline
$\rho_0$ $(d_1=1)$  &  $1$  & $1$   &  $1$    \\
\hline
$\rho_1$ $(d_2=1)$  &  $1$  & $e^{i\frac{2\pi}{3}}$   &  $e^{i\frac{4\pi}{3}}$   \\
\hline
$\rho_2$ $(d_3=1)$  &  $1$  & $e^{i\frac{4\pi}{3}}$   &  $e^{i\frac{2\pi}{3}}$    \\
\hline
\end{tabular}
\caption{The irreducible representations of the cyclic group $\Z_3$.}
\label{table-Z3}
\end{table}
Then, by using the same base matrices $M\in \{B,B(L),B(Q),B(S)\}$ as before (with $t=s^2$), we get the following matrices $\rho_i(M)$, for $i=0,1,2$, together with their spectra:

\begin{itemize}
\item
{\bf Adjacency spectrum.}
\begin{align*}
\rho_0(B) & =
\left(\begin{array}{cc}
0 & 1 \\
1 & 2 \end{array}\right),\ \Sp(\rho_0(B))=\{1\pm \sqrt{2}\};\\
\rho_1(B) & =\rho_2(B) = \left(\begin{array}{cc}
0 & 1  \\
1 & -1
\end{array}\right),\ \textstyle \Sp(\rho_1(B))=\Sp(\rho_2(B))=\{\frac{1}{2}(-1\pm \sqrt{5})\}.
\end{align*}

\item
{\bf Laplacian spectrum.}
\begin{align*}
\rho_0(L(B)) &=
\left(\begin{array}{cc}
 1 & -1 \\
-1 & 1 \end{array}\right), \ \Sp(\rho_0(L(B)))=\{0,2\}; \\
\rho_1(L(B)) &=\rho_2(L(B)) =
\left(\begin{array}{cc}
1 & -1  \\
-1 & 4
\end{array}\right),\ \textstyle \Sp(\rho_1(L(B)))=\Sp(\rho_2(L(B)))=\{\frac{1}{2}(5\pm \sqrt{13})\}.
\end{align*}

\item
{\bf Signess Laplacian spectrum.}
\begin{align*}
\rho_0(Q(B)) &=
\left(\begin{array}{cc}
1 & 1 \\
1 & 5 \end{array}\right),\ \Sp(\rho_0(Q(B)))=\{3\pm \sqrt{5}\};\\
\rho_1(Q(B)) & =\rho_2(Q(B)) = \left(\begin{array}{cc}
1 & 1  \\
1 & 2
\end{array}\right),\ \textstyle \Sp(\rho_1(Q(B)))=\Sp(\rho_2(Q(B)))=\{\frac{1}{2}(3\pm \sqrt{5})\}.
\end{align*}

\item
{\bf Seidel spectrum.}
\begin{align*}
\label{eq2:rhoB}
\rho_0(S(B)) &=
\left(\begin{array}{cc}
 2 & 1 \\
1 & -2 \end{array}\right),\ \Sp(\rho_0(S(B)))=\{\pm \sqrt{5}\};\\
\rho_1(S(B)) & =\rho_2(S(B)) = \left(\begin{array}{cc}
-1 & -2  \\
-2 & 1
\end{array}\right),\ \textstyle \Sp(\rho_1(S(B)))=\Sp(\rho_2(S(B)))=\{\pm\sqrt{5}\}.
\end{align*}
\end{itemize}

From these spectra, and applying \eqref{eq:t:spec-univ-ord}, we obtain again the different spectra of our example.

As a last comment, note that, in this context of ordinary voltage assignments on a group $G$ of order $n$,  of a base graph $\Gamma$, the base matrix of $J$ has entries
$(J)_{uv}=\sum_{g\in G} g$. Then, from representation theory, the matrices $\rho_0(J)$ and $\rho_i(J)$ for $i\neq 0$ turn out to be $nJ$ and $O$ (the $0$-matrix with appropriate dimension), respectively. Thus, for $i\neq 0$, $\rho_i(L(B))=-2\rho_i(A)-\rho_i(I)=-2\rho_i(A)-I$, and we can state the following lemma (see the above results for an example).
\begin{lemma}
Let $\Gamma$ be a base graph with  ordinary voltage assignment $\alpha$ on a group $G$ of order $n$.
Let $B$ be the base matrix of  $\Gamma$.
Then, apart from the eigenvalues of $nJ-2B-I$, the other Seidel eigenvalues of the lift $\Gamma^{\alpha}$ are of the form
$-2\lambda-1$, where $\lambda$ is an adjacency eigenvalue of $\Gamma^{\alpha}$ $($that is, $\lambda\in \Sp(\rho_i(B))$ for $i\neq 0$\/$)$.
\end{lemma}

\vskip 1cm

\noindent{\bf Acknowledgments.}~ The research of the two first authors is partially supported by
AGAUR from the Catalan Government under project 2017SGR1087 and by MICINN from the Spanish Government under project PGC2018-095471-B-I00. The research of the first author has also been supported by MICINN from the Spanish Government under project MTM2017-83271-R. The third and the fourth authors acknowledge support from the APVV Research Grants 15-0220 and 17-0428, and the VEGA Research Grants 1/0142/17 and 1/0238/19.

\end{document}